\documentclass[reqno]{amsart}

\usepackage[utf8]{inputenc}
\usepackage[OT2,T1]{fontenc}
\DeclareSymbolFont{cyrletters}{OT2}{wncyr}{m}{n}
\DeclareMathSymbol{\Sha}{\mathalpha}{cyrletters}{"58}
\usepackage{graphicx}
\graphicspath{ {./images/} }

\usepackage[hyphens,spaces,obeyspaces]{url}
\usepackage[colorlinks,allcolors=blue,hyperindex,breaklinks]{hyperref}
\hypersetup{
           breaklinks=true,   % splits links across lines
           colorlinks=true,   % displays links as colored text instead of blocks
        }

\usepackage{orcidlink}

\usepackage{pgfplots}
\pgfplotsset{compat=1.18}
\usepackage{mathrsfs}

\title[]{Covers of Bruhat--Tits trees}
\date{\today}
\usepackage[foot]{amsaddr}
\usepackage{comment}

\author[Corina-Gabriela Ciobotaru]{Corina-Gabriela Ciobotaru \orcidlink{0000-0002-1596-1061}}
\address{Department of Mathematics, Aarhus University, Ny Munkegade 118, Building 1530, 423, DK-8000
Aarhus C, Denmark}
\email{cociobotaru@math.au.dk}

\author[Peter Vang Uttenthal]{Peter Vang Uttenthal \orcidlink{0009-0001-0878-8213}}
\address{Department of Mathematics, Aarhus University, Ny Munkegade 118, Building 1530, 421, DK-8000
Aarhus C, Denmark}
\email{petervang@math.au.dk}

\newcommand{\GL}{\operatorname{GL}}

\newcommand{\Aut}{\operatorname{Aut}}

\newcommand{\Z}{\mathbb{Z}}

\newcommand{\e}{\stackrel{e}{\sim}}

\usepackage{amsthm,amsmath, mathrsfs, mathtools}
\usepackage{amsfonts}
\usepackage{amssymb}
\usepackage{fancyhdr}
\usepackage{IEEEtrantools}
\usepackage{tikz-cd}
\usepackage[english]{babel}
\usepackage[utf8]{inputenc}
\usepackage{csquotes}

\newtheorem{theorem}{Theorem}
\numberwithin{theorem}{section}
\newtheorem{lemma}{Lemma}
\numberwithin{lemma}{section}
\newtheorem{definition}{Definition}
\numberwithin{definition}{section}
\newtheorem{proposition}{Proposition}
\numberwithin{proposition}{section}
\newtheorem{corollary}{Corollary}
\numberwithin{corollary}{section}
\newtheorem{remark}{Remark}
\numberwithin{remark}{section}

\numberwithin{conjecture}{section}

\numberwithin{question}{section}

\numberwithin{example}{section}
\newtheorem{assumption}{Assumption}
\usepackage[hyphens,spaces,obeyspaces]{url}
\usepackage[colorlinks,allcolors=blue,hyperindex,breaklinks]{hyperref}

\hypersetup{
           breaklinks=true,   % splits links across lines
           colorlinks=true,   % displays links as colored text instead of blocks
        }
\usepackage{tabularx}
\usepackage{subcaption}
\begin{document}
\begin{abstract}
Let $G$ be a locally compact group and let $\widetilde G$ be a central extension of $G$ that splits over a maximal compact subgroup $K$ of $G$. We derive an explicit cocycle that lifts the natural action of $G$ on the homogeneous space $G/K$ to an action of $\widetilde G$. As an application, for a non-Archimedean local field $F$, we construct a connected locally finite tree on which the metaplectic covers of $\mathrm{GL}_2(F)$ act by automorphisms, 
providing a geometric 
analog of the Bruhat--Tits tree of $\mathrm{GL}_2(F)$. 
Furthermore, under suitable transitivity assumptions, we prove that $(\widetilde G,\widetilde K)$ is a Gelfand pair. Finally, we describe the associated parabolic and contraction subgroups with respect to $\widetilde G$  from the perspective of the geometry of the constructed tree.
\end{abstract}

\maketitle

\tableofcontents

\section{Introduction}

The geometry of Bruhat--Tits buildings plays a central role in the study of reductive groups over non-Archimedean local fields. In the rank-one case, the Bruhat--Tits tree of $\mathrm{GL}_2(F)$ provides a geometric model that encodes algebraic and dynamical properties of the group which, in turn, are useful for classifying its unitary representations. 
For metaplectic groups and more general central extensions, however, it is less clear what the corresponding geometric object should be.
Yet, 
covers of Bruhat--Tits buildings
have shown to be applicable in number theory. For instance, the need for passing to a double cover of a Bruhat--Tits tree arises 
in the work of 
A. Ash and D. Doud \cite{ash-doud} on attaching even Galois representations
induced from totally real quadratic fields 
to Hecke eigenfunctions in the cohomology of arithmetic groups. 

The goal of the present work is  
to construct geometric models naturally associated with covering groups, 
which requires bringing techniques from the
rather different research backgrounds of the two authors together.

Our starting point is a locally compact group $G$, a maximal compact subgroup $K$ of $G$, and a central extension
\[
1 \longrightarrow T \longrightarrow \widetilde G \longrightarrow G \longrightarrow 1
\]
determined by a cohomology class $\alpha \in H^2(G,T)$ with values in an abelian locally compact topological group $T$. Given the homogeneous space $X=G/K$ and a cross-section $\theta:X\to G$, we form 
the standard cocycle
\begin{equation} \label{sigmaintro}
\sigma(g,x)=\theta(gx)^{-1}g\theta(x).
\end{equation}
The first result shows that when the extension $\widetilde G $ splits over $K$, 
the cocycle $\sigma$ admits a natural lift 
to a cocycle $\sigma^{(\alpha)}$ 
on $\widetilde{G} \times X$
that depends on $\alpha$ and $\theta$
via an explicit formula \eqref{j} given below.

\begin{theorem}\label{thm1intro} Let $G$ be a locally compact group 
acting on an abelian group $T$. Fix $\alpha \in H^2(G,T)$
and let $K$ be maximal compact subgroup of $G$
that splits over $\widetilde{G}$.
For $X \simeq G/K$, a section
$\theta: X \to G$, and $\sigma$ given in \eqref{sigmaintro}, the function
\begin{equation} \label{j}
\sigma^{(\alpha)}(\tilde{g}, x)
    := \alpha\!\left(g \theta(x),\, \sigma(g,x)^{-1}\right)
       \alpha\!\left(g, \theta(x)\right) t
       \quad\quad \tilde{g} = (g,t)
\end{equation}
on $\widetilde{G} \times X$ defines a cocycle 
that induces an action of $\widetilde{G}$ on 
$\widetilde{X} := X \rtimes_{\sigma^{(\alpha)}} T$.
\end{theorem}
\noindent The formula \eqref{j} 
follows purely from group cohomology.
As a tool for lifting locally compact group actions on homogeneous spaces, 
we illustrate the use of Theorem \ref{thm1intro} with an application to Bruhat--Tits trees $\mathscr{T}$ over non-Archimedean local fields $F$.
For each integer $m\geq 1$ and cohomology class
$\alpha \in H^2(\GL_2(F) ,\mu_m(F))$
valued in the roots of unity of the field,
we use \eqref{j} to construct an $m$-fold cover
\[
\mathscr{T}^{(m)} := \mathscr{T} \rtimes_{\sigma^{(\alpha)}} \mu_m(F)
\]
that we equip with the structure of a connected tree on which the $m$-fold covering group $\GL_2(F)^{(m)}$ of $\GL_2(F)$ acts by automorphisms in the category of locally finite connected trees.
In the setting of metaplectic groups,
the connected tree $\mathscr{T}^{(m)}$ provides a geometric analog of the Bruhat--Tits tree $\mathscr{T}$ for $\GL_2(F)$.

Once a geometric model with a $\widetilde{G}$-action is fixed, we may investigate some of the features that are relevant for harmonic analysis on $\widetilde{G}$. 
In Section~\ref{Gelfand_pairs} we study Gelfand pairs arising from the action of covering groups on locally finite, infinite trees. Assuming a suitable polar decomposition and transitivity of the compact stabilizer on the boundary of the tree, we prove that $(\widetilde G,\widetilde K)$ is a Gelfand pair. 

\begin{theorem} \label{GelfandIntro}
Let $G$ be a locally compact group and 
let $X$ be a locally finite, infinite tree with a transitive $G$-action. 
Let $L$ be a geodesic line in $X$, let $x \in L$, and fix a maximal compact subgroup $K=\operatorname{Stab}_G(x)$.
Assume that $K$ acts transitively on the virtual boundary $\partial X$
and that the $G$--orbit through $x\in L$ is equal to $L$.
Then $(\widetilde{G},\widetilde{K})$ is a Gelfand pair. 
\end{theorem}

The proof is geometric in nature and yields an alternative approach to Gelfand-pair results for central extensions of locally compact groups. 

Finally, we study contraction and parabolic subgroups associated with lifts of translation elements in $G$ to covering groups. We describe these subgroups in terms of the geometry of the tree $\widetilde X$ and relate them to stabilizers of ends of the tree. 
This provides a geometric interpretation of certain structural features of metaplectic groups. 

%The paper is organized as follows. Section~3 develops the cocycle lifting construction. Section~4 constructs the covering tree and the corresponding group action. Section~5 establishes the Gelfand-pair property. Finally, Section~6 studies contraction and parabolic subgroups associated with the covering group.

\subsection{Acknowledgments}
This work was supported by VILLUM FONDEN through grants VIL53023 (to C.C.) and VIL54509 (to P.V.U.).

\section{Formula for a twisted cocycle}
\label{twisted_cocycle}

Let $G$ be a locally compact topological group and let $K \leq G$ be a compact subgroup. Let $T$ be an abelian locally compact topological group, not necessarily contained in $G$. 

\subsection{Central extensions defined by a cocycle}

Suppose we are given a map
\[
\alpha : G \times G \to T
\]
satisfying the cocycle identity
\begin{equation}\label{equ:cocycle_iden}
\alpha(g_1 g_2, g_3)\,\alpha(g_1, g_2)
    \;=\;
\alpha(g_1, g_2 g_3)\,\alpha(g_2, g_3),
\end{equation}
for all $g_1,g_2,g_3\in G$, and normalized so that
\[
\alpha(g,1) = \alpha(1,g) = 1 \qquad \text{for all } g\in G.
\]

Then $\alpha$ is an element of $H^2(G,T)$ and defines a group structure on the set
\begin{equation}\label{equ:group_structure}
\widetilde{G} := G \times T \text{ that is given by } (g_1,t_1)\,(g_2,t_2): = (g_1 g_2,\; \alpha(g_1,g_2)\, t_1 t_2).
\end{equation}

We denote this extension by
\[
\widetilde{G} = G \times_{\alpha} T.
\]

\subsection{Splitting over a subgroup}
\begin{definition}
    We say that a closed subgroup $H$ of $G$ splits in $\widetilde{G}$ if there exists a map $s : H \to T$ such that
\[
\alpha(h_1,h_2)
    = s(h_1 h_2)\, (s(h_1)\, s(h_2))^{-1},
    \qquad \forall\, h_1,h_2\in H.
\]
\end{definition}
If $\widetilde{G}$ splits over $H$, one may verify using the group structure from (\ref{equ:group_structure}) that the map
\[
h \longmapsto (h, s(h))
\]
is a group homomorphism from $H$ to $\widetilde{G}$ that is, in addition, injective.

\begin{remark}
Let $\alpha \in H^2(G,T)$. 
A closed subgroup $H$ splits in $\widetilde{G} = G\times_\alpha T$ if and only if $\alpha$ belongs to 
the kernel of the restriction map
$\operatorname{res}: H^2(G, T)  \to H^2(K,T).$
\end{remark}

Examples and non-examples of subgroups of $G$ that split in $\widetilde{G}$ are given at the beginning of Section~\ref{sec::connected_tree} below.

\begin{remark}
Let $F$ be a non-Archimedean local field of characteristic zero with ring of integers $\mathcal{O}_F$. 
Define $G:=\GL_2(F)$ with a maximal compact subgroup $K := \GL_2(\mathcal{O}_F)$ and choose $T = \Z/2\Z$.
For all integers $N \equiv 0 \bmod 4$, the closed subgroup
\[
K^N :=\left \{ \begin{psmallmatrix} a & b \\ c & d \end{psmallmatrix} \in K \mid a \equiv 1 \bmod N, c\equiv 0 \bmod N \right \}
\]
splits in $\widetilde{G}$, and if $F$ has odd residual characteristic, then $K^N = K$ \cite[pp. 17-18]{gelbart}.  
\end{remark}

In this article, we will assume that the compact subgroup $K$ of $G$ splits in $\widetilde{G}$. 
After modifying $\alpha$ by a cohomologous cocycle using the splitting $s: K \to T$, we may and do assume that
\[
\alpha|_{K \times K} = 1.
\]

\subsection{Cross-sections and the associated cocycle}

Consider the exact sequence
\[
1 \to K \to G \xrightarrow{\;\pi\;} G/K \to 1.
\]
A cross-section is a map 
\[
\theta : G/K \to G \text{ such that }\theta(x)\,K = x.
\]
To such a cross-section one associates the map
\begin{equation}\label{sig}
\sigma : G \times (G/K) \to K \text{ defined by }
\sigma(g,x)
    := \theta(gx)^{-1}\, g\, \theta(x) \in K.
\end{equation}
It is easy to verify that $\sigma$ is well defined, takes values in $K$, and satisfies the cocycle identity.

%Let $G$ be a locally compact topological group with a maximal compact subgroup $K$. Let $T$ be a subgroup of a torus in $G$.
%Suppose $\alpha: G\times G \to T$ satisfies the cocycle identity
%\begin{equation}
%\label{equ::cocycle_iden}
 %   \alpha(g_1 g_2, g_3) \alpha(g_1, g_2) = \alpha(g_1, g_2 g_3 ) \alpha(g_2,g_3)
%\end{equation}
%for all $g_1,g_2,g_3\in G$ and that 
%$\alpha(g,1) = \alpha(1,g) = 1$ for all $g\in G$. Then %$\alpha \in H^2(G,T)$ defines a group law on the set  %$\tilde{G} = G\times T$ by
%\[
%(g_1,t_1)(g_2,t_2) = (g_1g_2, \alpha(g_1,g_2) t).
%\]
%We may write $\tilde{G} = G \times_\alpha T$ to emphasize the dependence on $\alpha \in H^2(G,T)$. Fix some maximal compact $K$ in $G$ that splits over $\tilde{G}$ in the sense that there is a map $s: K \to T$ such that $\alpha(k_1,k_2)= (s(k_1k_2))^{-1}s(k_1)s(k_2)$ for all $k_1,k_2\in K$.
%Then  $k \mapsto (k,s(k))$ is a homomorphism that embeds $K$ in $\tilde{G}$. Without loss of generality, we will assume that $\alpha|_{K\times K } = 1$.

%A cross-section for the sequence  $1\to K \to G \to G/K \to 1$ is a map $$\theta: G/K \to G \text{ such that } \theta(x) K = x,$$ for all $x\in G/K$. With it there is associated the cocycle $$\sigma: G \times G/K \to K \text{ given by } (g,x) \mapsto \sigma(g,x): = \theta (gx)^{-1} g\theta (x) \in K.$$

\begin{theorem} \label{twisted}
Keeping the notation introduced above, let 
\(\widetilde{G} = G \times_{\alpha} T\), and fix a cross-section 
\(\theta : G/K \to G\) with associated cocycle \(\sigma\) given in 
\eqref{sig}. Define
\[
\sigma^{(\alpha)} : \widetilde{G} \times (G/K) \to T \quad  \text{ by } \quad
\sigma^{(\alpha)}(\tilde{g}, x)
    := \alpha\!\left(g \theta(x),\, \sigma(g,x)^{-1}\right)
       \alpha\!\left(g, \theta(x)\right) t,
\]
for \(\tilde{g}: = (g,t) \in \widetilde{G}\) and \(x \in G/K\).

Then \(\sigma^{(\alpha)}\) satisfies the cocycle identity
\[
\sigma^{(\alpha)}(\tilde{g}_1,\, g_2 x)\;
\sigma^{(\alpha)}(\tilde{g}_2,\, x)
    = \sigma^{(\alpha)}(\tilde{g}_1 \tilde{g}_2,\, x),
\]
for all \(x \in X\) and all \(\tilde{g}_k = (g_k, t_k) \in \widetilde{G}\),
\(k = 1,2\).
\end{theorem}

We will refer to $\sigma^{(\alpha)}$ 
as the $\alpha$-twisted cocycle obtained from $\sigma$. 

\begin{corollary}
\label{cor::corollary_twisted_action}
Let $X:= G/K$. Then the formula for $\tilde{g} \in \widetilde{G}$, $x \in G/K$, and $\xi \in T$,
\[
\tilde{g} (x,\xi) 
:= \left( gx, \sigma^{(\alpha)}( \tilde{g} , x ) \xi 
\right)
\]
defines an action of $\widetilde{G}$ on the space
$
\widetilde{X} := X \times_{\alpha,\theta} T
$
that extends the natural action of $G$ on $X$
by giving a commutative diagram
\begin{center}
\begin{tikzcd}
    \widetilde{G}\times \widetilde{X} \arrow[r, dashed] \arrow[d] & \widetilde{X} \arrow[d] \\
    G\times X \arrow[r] & X
\end{tikzcd}
\end{center}
whose vertical arrows are the canonical projections. 
\end{corollary}

\begin{proof}[Proof of Theorem \ref{twisted}]
%The cocycle identity  (\ref{equ::cocycle_iden}) for $\alpha$, taking $g_1: = g \theta(x), g_2 :=\sigma(g,x)^{-1}$, and $g_3:= \sigma(g,x)$,  implies: 
%\begin{equation*}
   % \alpha( g \theta(x) ,\sigma(g,x)^{-1} \sigma(g,x) ) \alpha(  %\sigma(g,x)^{-1} , \sigma(g,x)  ) =
%\alpha(g \theta(x)   \sigma(g,x)^{-1} , \sigma(g,x) )\alpha( g \theta(x), \sigma(g,x)^{-1} ).
%\end{equation*}
%Since $\alpha$ is trivial on $K\times K$, $\sigma(g,x) \in K$, and $\alpha(g,1 ) = 1$ for all $g\in G$, the above equality is equivalent to
%\begin{equation}
%\label{equ::first_eq}
%\begin{split}
   % 1 &=
%\alpha(g \theta(x)  \sigma(g,x)^{-1} , \sigma(g,x) )\alpha( g \theta(x), %\sigma(g,x)^{-1} )\\
%& = \alpha(\theta(gx), \sigma(g,x) )\alpha( g \theta(x), \sigma(g,x)^{-1} ).
%\end{split}
%\end{equation}

%By taking $g_1: = g, g_2 :=\theta(x)$, and $g_3:= \sigma(g,x)^{-1}$ in the cocycle identity (\ref{equ::cocycle_iden}) we get:
%\begin{align}
%\label{equ::second_eq}
 %   \alpha(g \theta(x), \sigma(g,x)^{-1}) \alpha(g, \theta(x)) &= \alpha(g, \theta(x) \sigma(g,x)^{-1}) \alpha( \theta(x) ,  \sigma(g,x)^{-1}) 
%\end{align}
%for all $g \in G$ and $x\in X$. \red{Where do you use (\ref{equ::second_eq})?}
%In particular, using that $\alpha$ is trivial on $K\times K$ and that $\alpha(g,1) = 1$ for all $g\in G$, 
%\[
%1 = \alpha(g \theta(x)   \sigma(g,x)^{-1} , \sigma(g,x) )\alpha( g \theta(x), \sigma(g,x)^{-1} ).
%\]

Let $X:= G/K$.
We prove the theorem by showing that for any 
\(\tilde{g} = (g,t) \in \widetilde{G}\) and any \(x \in X\), the assignment
\begin{equation}
\label{equ::cocycle_sigma_tilde}
\sigma'(\tilde{g}, x)
    := \bigl(\sigma(g,x),\;
        \alpha(g \theta(x),\, \sigma(g,x)^{-1})
        \alpha(g, \theta(x))\, t \bigr)
        \in K \times T
\end{equation}
defines a map \(\sigma' : \widetilde{G} \times X \to K \times T\)
which moreover satisfies the cocycle identity
\[
\sigma'(\tilde{g}_1,\, g_2 x)\;
\sigma'(\tilde{g}_2,\, x)
    = \sigma'(\tilde{g}_1 \tilde{g}_2,\, x),
\]
for all \(x \in X\) and all \(\tilde{g}_1=(g_1, t_1), \tilde{g}_2 =(g_2,t_2) \in \widetilde{G}\).
%This establishes that \(\sigma'\) is indeed a cocycle on \(\widetilde{G} \times X\) with values in \(K\times T\).

\medskip
To proceed, we first show that \(\sigma'\) is the map naturally associated with the lifted cross–section
\[
\theta' : X \longrightarrow G \times \{1\} \subset \widetilde{G}, 
\qquad 
\theta'(x) := (\theta(x), 1).
\]
Thus \(\theta'\) may be regarded as a cross–section of the homogeneous space \(X = G/K\), but taking values in the enlarged group \(\widetilde{G}\).

\textbf{Claim 1}. We have that
$$\sigma'(\tilde{g}, x) =  \theta' \left(\tilde{g} x \right)^{-1} \tilde{g} \theta'(x).$$ 
Indeed, we impose that \(\widetilde{G}=G\times T\) acts on \(X=G/K\) via its projection onto
\(G\), so we define that \(\tilde{g}x=(g,t)x: = g x\). Thus,
\[
\theta'(\tilde{g}x)
    = (\theta(gx),1)
    \quad\Longrightarrow\quad
\theta'(\tilde{g}x)^{-1}
    = \bigl(\theta(gx)^{-1},\,
        \alpha(\theta(gx),\theta(gx)^{-1})^{-1}\bigr).
\]
Similarly,
\[
\tilde{g}\,\theta'(x)
    = (g,t)\,(\theta(x),1)
    = \bigl(g\theta(x),\, \alpha(g,\theta(x))\, t\bigr).
\]
Multiplying these two expressions in \(\widetilde{G}\) yields the identity:
\begin{equation*}
    \begin{split}
        \theta' \left(\tilde{g} x \right)^{-1} \tilde{g} \theta'(x)& = (\theta(gx)^{-1}, \alpha(\theta(gx), \theta(gx)^{-1})^{-1}) (g \theta(x), \alpha(g, \theta(x))t)\\
        & = (\sigma(g, x), \alpha(\theta(gx)^{-1},g \theta(x)) \alpha(\theta(gx), \theta(gx)^{-1})^{-1} \alpha(g, \theta(x))t).
    \end{split}
\end{equation*}
It remains to verify that we indeed obtain the second term of (\ref{equ::cocycle_sigma_tilde}), which is equivalent to verifying:
\begin{equation}
\label{equ::eq_to_verify}
\begin{split}
    &\alpha(\theta(gx)^{-1},g \theta(x)) \alpha(\theta(gx), \theta(gx)^{-1})^{-1} = \alpha(g \theta(x) ,\sigma(g,x)^{-1}) \\
&\Longleftrightarrow \alpha(\theta(gx)^{-1},g \theta(x)) = \alpha(\theta(gx), \theta(gx)^{-1})  \alpha(g \theta(x) ,\sigma(g,x)^{-1}).
\end{split}
\end{equation}

Indeed, taking $g_1: = \theta(gx), g_2 :=g_1^{-1}$ and $g_3:= g\theta(x)$ in the cocycle identity (\ref{equ:cocycle_iden}) we get:
\begin{equation}
\label{equ::second_eqq}
\begin{split}
   \alpha(\theta(gx), \theta(gx)^{-1}) &=  \alpha(\theta(gx), \theta(gx)^{-1}) \alpha(1,g\theta(x))\\
   &= \alpha(\theta(gx), \sigma(g,x)) \alpha( \theta(gx)^{-1} ,g\theta(x)).
\end{split}
\end{equation}
Using (\ref{equ::second_eqq}), the last term of (\ref{equ::eq_to_verify}) becomes
\begin{equation}
\label{equ::second_eqqq}
\begin{split}
   \alpha(\theta(gx), \theta(gx)^{-1})  \alpha(g \theta(x) ,\sigma(g,x)^{-1})
   &= \alpha(\theta(gx), \sigma(g,x)) \alpha( \theta(gx)^{-1} ,g\theta(x)) \alpha(g \theta(x) ,\sigma(g,x)^{-1})\\
   &= \alpha( \theta(gx)^{-1} ,g\theta(x)) \alpha(\theta(gx), \sigma(g,x)) \alpha(g \theta(x) ,\sigma(g,x)^{-1})\\
   &= \alpha( \theta(gx)^{-1} ,g\theta(x))\\
\end{split}
\end{equation}
which is exactly what we wanted in (\ref{equ::eq_to_verify}).
The last equality of \eqref{equ::second_eqqq} follows from the cocycle identity \eqref{equ:cocycle_iden} with $g_1 := g \theta(x), g_2:= \sigma(g,x)^{-1}, g_3:= \sigma(g,x)$, and since $\alpha$ is trivial on $K\times K$, $\sigma(g,x)= \theta(gx)^{-1}\, g\, \theta(x)  \in K$, and $\alpha(g,1) = 1$:
$$\alpha(\theta(gx), \sigma(g,x)) \alpha(g \theta(x) ,\sigma(g,x)^{-1}) = \alpha(g \theta(x) ,1) \alpha(\sigma(g,x)^{-1} ,\sigma(g,x))=1.$$
This finish the proof of \textbf{Claim 1}.
\medskip

\noindent The final step is to verify the cocycle identity for $\sigma'$. 
On the one hand, we have the formula 
\begin{equation}\label{g_1g_2}
\sigma'(\tilde{g}_1\tilde{g}_2, x)
= ( \sigma(g_1g_2,x), \alpha(g_1 g_2 \theta x, \sigma(g_1g_2,x)^{-1})
\alpha(g_1 g_2, \theta x) 
\alpha(g_1, g_2 ) t_1t_2 ).
\end{equation}
On the other hand, we use  \eqref{equ::second_eqqq} to deduce that
\begin{align*}
\sigma'(\tilde{g}_1,  \tilde{g}_2  x)\sigma'(\tilde{g}_2, x) &=
    (\sigma(g_1,  g_2x), \frac{ \alpha(  (\theta  g_1 g_2 x)^{-1},  g_1 \theta  g_2x  )\alpha(g_1, \theta  g_2x) }{ \alpha(\theta  g_1 g_2x,  (\theta g_1 g_2 x)^{-1} )}t_1) \\
    &\times
    (\sigma(g_2,x), \frac{ \alpha(  (\theta  g_2x)^{-1},  g_2 \theta  x  )\alpha(g_2, \theta  x) }{ \alpha(\theta  g_2x,  (\theta g_2x)^{-1} )}t_2)\\
    &= 
    \bigg( \sigma(g_1g_2,x), \alpha( \sigma(g_1, g_2 x), \sigma  (g_2, x)) \\
&\times 
\frac{ \alpha(  (\theta  g_1 g_2 x)^{-1},  g_1 \theta  g_2x  )\alpha(g_1, \theta  g_2x) }{ \alpha(\theta  g_1 g_2x,  (\theta g_1 g_2 x)^{-1} )}  \\
&\times 
    \frac{ \alpha(  (\theta  g_2x)^{-1},  g_2 \theta  x  )\alpha(g_2, \theta  x) }{ \alpha(\theta  g_2x,  (\theta g_2x)^{-1} )} t_1 t_2 \bigg) 
\end{align*}
Then, 
we use the identity $
    \theta(g_1g_2x) = 
     g_1 g_2\theta(x) \sigma(g_1g_2, x)^{-1}$
for $g_1,g_2\in G$ and $x\in X$
to rewrite the coordinate in $T$ as
\begin{IEEEeqnarray*}{l}
\alpha( \sigma(g_1, g_2 x), \sigma  (g_2, x))  
\frac{ \alpha( \sigma(g_1g_2,x) (\theta x)^{-1} g_2^{-1} g_1^{-1} ,  g_1 \theta  g_2x  )\alpha(g_1, \theta  g_2x)}{ \alpha(g_1 g_2  \theta  (x)  \sigma(g_1g_2, x)^{-1}, \sigma(g_1g_2, x) (\theta  x)^{-1} g_2^{-1} g_1^{-1})} \\
\times\frac{ \alpha(  (\theta  g_2x)^{-1},  g_2 \theta  x  )\alpha(g_2, \theta  x) }{ \alpha(\theta  g_2x,  (\theta g_2x)^{-1} )} t_1 t_2  \\
= \alpha( \sigma(g_1, g_2 x), \sigma  (g_2, x)) \\ 
\times
  \frac{\alpha( \sigma(g_1g_2,x) (\theta x)^{-1} g_2^{-1} g_1^{-1} , 
 g_1 g_2 \theta  x  
%  \sigma(g_1g_2,x)^{-1} 
%\sigma(g_1, g_2x)\sigma(g_2,x)
%\sigma(g_2,x)^{-1}
%= \sigma(g_1g_2,x)^{-1} \sigma(g_1, g_2x)
\sigma(g_1g_2,x)^{-1} \sigma(g_1, g_2x) )}{\alpha(  \sigma(g_1g_2,x)(\theta  x)^{-1} g_2^{-1} g_1^{-1}, g_1 g_2  \theta  x  \sigma(g_1g_2,x)^{-1} )}   \\ 
\times
\frac{ \alpha(g_1, g_2\theta x \sigma(g_2,x)^{-1}) \alpha(  \sigma(g_2,x) (\theta x)^{-1} g_2^{-1},  g_2 \theta  x  )\alpha(g_2, \theta  x)}{  \alpha(   g_2\theta x \sigma(g_2,x)^{-1} ,\sigma(g_2,x) (\theta x)^{-1} g_2^{-1} ) } t_1t_2 \\
= \alpha( \sigma(g_1, g_2 x), \sigma  (g_2, x)) \\ 
\times
\frac{ \alpha(g_1, g_2\theta x \sigma(g_2,x)^{-1}) \alpha(  \sigma(g_2,x) (\theta x)^{-1} g_2^{-1},  g_2 \theta  x  )\alpha(g_2, \theta  x)}{ 
\alpha(  g_1g_2\theta(x) \sigma(g_1g_2,x)^{-1}, \sigma(g_1,g_2x)) \alpha(   g_2\theta x \sigma(g_2,x)^{-1} ,\sigma(g_2,x) (\theta x)^{-1} g_2^{-1} ) }  
     t_1 t_2 \\
=\alpha( \sigma(g_1, g_2 x), \sigma  (g_2, x)) \\ \times
\frac{ \alpha(g_1, g_2\theta (x) \sigma(g_2,x)^{-1}) 
\alpha(g_2, \theta  x)}{ 
\alpha(  g_1g_2\theta(x) \sigma(g_1g_2,x)^{-1}, \sigma(g_1,g_2x) ) 
\alpha(   g_2\theta x \sigma(g_2,x)^{-1} , \sigma(g_2,x) ) }  
     t_1 t_2 \\
 =    \frac{\alpha( \sigma(g_2, x)^{-1}, \sigma  (g_2, x) ) }{\alpha(   g_2\theta(x) \sigma(g_2,x)^{-1} , \sigma(g_2,x) ) }
\frac{ \alpha(g_1, g_2\theta (x) \sigma(g_2,x)^{-1}) 
\alpha(g_2, \theta  x)}{ 
\alpha(  g_1g_2\theta(x) \sigma(g_1g_2,x)^{-1} , \sigma(g_1,g_2x) )  }  t_1t_2\\
=\alpha(g_2 \theta x ,   \sigma(g_1g_2,x)^{-1} \sigma(g_1, g_2x) )) 
\frac{ \alpha(g_1, g_2\theta (x) \sigma(g_1g_2,x)^{-1} \sigma(g_1, g_2x) )  
}{ 
\alpha(  g_1g_2\theta(x) \sigma(g_1g_2,x)^{-1}, \sigma(g_1,g_2x) )  }  \alpha(g_2, \theta  x) t_1t_2\\
= \alpha(g_2 \theta x ,   
     \sigma(g_1g_2,x)^{-1} \sigma(g_1, g_2x) )
     ) 
\frac{ 
\alpha(g_1, 
g_2 \theta (x) \sigma(g_1g_2,x)^{-1} )
}{
\alpha(g_2 \theta (x) \sigma(g_1g_2,x)^{-1}, \sigma(g_1,g_2x)
 }
\alpha(g_2, \theta  x) t_1t_2 \\
=\frac{\alpha(g_2 \theta x,  \sigma(g_1g_2,x)^{-1})}{\alpha(\sigma(g_1g_2,x)^{-1}, \sigma(g_1,g_2x) )}    
\alpha(g_1, g_2 \theta (x) \sigma(g_1g_2,x)^{-1} )
\alpha(g_2, \theta  x) t_1t_2.
\end{IEEEeqnarray*}
Since $\sigma$ is valued in $K$ and $\alpha$
is trivial on $K\times K$, the last expression
reduces to 
\begin{IEEEeqnarray*}{l}
\alpha(g_2 \theta x,  \sigma(g_1g_2,x)^{-1})   \alpha(g_1, g_2 \theta (x) \sigma(g_1g_2,x)^{-1} )
\alpha(g_2, \theta  x) t_1t_2\\
=\alpha(g_1 g_2 \theta x, \sigma(g_1g_2,x)^{-1})\alpha(g_1, g_2\theta x )  \alpha(g_2, \theta x) t_1t_2\\
=\alpha(g_1 g_2 \theta x, \sigma(g_1g_2,x)^{-1})
\alpha(g_1g_2, \theta  x )\alpha(g_1,  g_2) t_1t_2.
\end{IEEEeqnarray*}
We conclude that  
\begin{align*}
\sigma'(\tilde{g}_1, \tilde{g}_2x) \sigma'( \tilde{g}_2, x) 
&= ( \sigma(g_1g_2,x), \alpha(g_1 g_2 \theta x, \sigma(g_1g_2,x)^{-1})
\alpha(g_1 g_2, \theta x) 
\alpha(g_1, g_2 ) t_1t_2 ) \\
&= \sigma'(\tilde{g}_1\tilde{g}_2, x).    
\end{align*}
\end{proof}

\begin{remark}
   Notice that
\[
\tilde{g}  . (x,  (k,\xi) ) = (g,t)  . (x,  (k,\xi) ) := 
(g x,
 \sigma'(\tilde{g},x) (k,\xi) ) 
\]
defines an action of 
$\tilde{G}$ on $G/K \times (K\times T) $ 
since
\begin{align*}
(g_1,t_1).( (g_2,t_2). (x,(k,\xi) )  &=
(g_1,t_1).  ( g_2 x, \sigma'(\tilde{g}_2,x) (k,\xi) ) \\
&= (g_1g_2x, 
\sigma'( \tilde{g}_1 , g_2  x )
\sigma'(\tilde{g}_2,x)  (k,\xi) )\\
&= (g_1g_2x,\sigma'(\tilde{g}_1 \tilde{g}_2,\, x)(k,\xi)).
\end{align*}
%In addition, the formula
%\[
%(g,t)  . (x,  \xi)  = (gx,\alpha( g\theta(x), \sigma(g,x)^{-1} )\alpha(g, \theta(x))t\xi ) 
%\]
%defines a $\tilde{G}$-action on $G/K \times T$. 
\end{remark}

\section{A connected tree for metaplectic $\GL_2(F)$}
\label{sec::connected_tree}

Let $F$ be a non-Archimedean local field and $A$ an abelian topological group.
Central extensions
\[
1 \to A \to \widetilde{\GL}_n(F) \to \GL_n(F) \to 1
\]
form a large class for $n \ge 2$.

A distinguished subclass of central extensions of $\mathrm{GL}_n(F)$ by $A$ is given
by those arising from \emph{Steinberg cocycles}. More precisely, for $n \ge 2$,
let $\mathrm{E}_n(F)$ denote the \emph{elementary subgroup} of
$\mathrm{GL}_n(F)$, that is, the subgroup generated by the elementary matrices
\[
E_{ij}(r) = I_n + r e_{ij}, \qquad i \neq j,\; r \in F,
\]
where $e_{ij}$ denotes the matrix with a $1$ in the $(i,j)$-entry and zeros elsewhere.

For each $n \ge 2$, the \emph{Steinberg group} $\mathrm{St}_n(F)$ is defined by
generators $X_{ij}(r)$, subject to the Steinberg relations, and there is a natural surjective homomorphism
\[
\varphi_n \colon \mathrm{St}_n(F) \longrightarrow \mathrm{E}_n(F)
\subset \mathrm{GL}_n(F)
\]
sending $X_{ij}(r)$ to $E_{ij}(r)$. Set $\mathrm{St}(F)$ and $\mathrm{GL}(F)$ to be the colimits of the groups
$\mathrm{St}_n(F)$ and $\mathrm{GL}_n(F)$, respectively. The homomorphisms
$\varphi_n$ induce a group homomorphism
\[
\varphi \colon \mathrm{St}(F) \longrightarrow \mathrm{GL}(F),
\]
whose image is the colimit $\mathrm{E}(F)$ of the groups $\mathrm{E}_n(F)$. Define $K_2(F) := \ker(\varphi)$. The group $K_2(F)$ is abelian and coincides with the center of $\mathrm{St}(F)$; moreover, it is independent of $n$. The group $K_2(F)$ measures precisely how much the lifts of commuting elementary matrices fail to commute in the Steinberg group $\mathrm{St}_n(F)$.

For $n \ge 3$, the group $\mathrm{E}_n(F)$ is perfect, and the extension
\[
1 \longrightarrow K_2(F) \longrightarrow \mathrm{St}_n(F)
\longrightarrow \mathrm{E}_n(F) \longrightarrow 1
\]
is the universal central extension of $\mathrm{E}_n(F)$. In particular, for
any abelian group $A$, central extensions of $\mathrm{E}_n(F)$ by $A$ are in
natural bijection with homomorphisms $K_2(F) \to A$, or equivalently with
cohomology classes in $H^2(\mathrm{E}_n(F),A)$. All such extensions arise from the \emph{Steinberg cocycles} with values in $A$.

In contrast, for $n = 2$, the group $\mathrm{E}_2(F) = \mathrm{SL}_2(F)$ is not
perfect, and $\mathrm{St}_2(F)$ is no longer a universal central extension. Nevertheless, Steinberg cocycles with values in $A$ still define a distinguished subclass of central extensions of $\mathrm{SL}_2(F)$, which correspond to quotients of the Steinberg central extension $\mathrm{St}_2(F) \to \mathrm{SL}_2(F)$.

Across all $n \ge 2$, the above framework encompasses 
the classical metaplectic covers, all Brylinski--Deligne (BD) covers, and their higher-degree analogs. These extensions are completely characterized by splitting uniquely over the unipotent subgroups of $\mathrm{E}_n(F)$. While the diagonal 
elements of the base group themselves commute, their formal lifts to the extension group generally fail to do so. The resulting extension class is entirely dictated 
by the commutators of these lifts, which define the Steinberg symbols and measure the failure of the set-theoretic section to preserve commutativity. Consequently, these symbols identically satisfy the defining relation $\{a,1-a\}=1$ (or $0$ 
in additive notation).

Moreover, it is known that Steinberg cocycles for $E_n(F)$ do not exhaust all central extensions of $\GL_n(F)$. Moreover,
when $A=\mu_m(F)\subset F^\times$ is the group of $m$th roots of unity in the local field $F$, for a fixed integer $m$, one obtains the Brylinski--Deligne (BD) covers of $\GL_n(F)$, which form a further restricted subclass of Steinberg extensions: they are exactly those arising from Weyl-invariant quadratic forms on the cocharacter lattice $X_*(T)$ of a maximal torus $T \subset \GL_n$. Consequently, one has strict inclusions
\[
\text{BD covers}
\;\subsetneq\;
\text{Steinberg cocycles}
\;\subsetneq\;
\text{all central extensions}.
\]

This hierarchy explains why splitting over the hyperspecial maximal compact subgroup $\GL_n(\mathcal O_F)$ is exceptional and occurs only in special situations, such as the tame standard metaplectic cover. More precisely, additional constraints arise from the interaction between the local field $F$ and the degree $m$ of the metaplectic cover.

For the remainder of this section, we impose the following assumptions. Let
\[
A = \mu_m(F) \subset F^\times
\]
be the group of $m$-th roots of unity in the local field $F$, and let $\widetilde{\GL}_n(F)$ denote the corresponding Brylinski--Deligne (BD) cover of $\GL_n(F)$.

We work in the \emph{tame case}, meaning that:
\begin{itemize}
    \item $m$ is coprime to the residue characteristic $q$ of $F$, and
    \item $q \equiv 1 \pmod{2n}$.
\end{itemize}

Under these assumptions, it is known that the $m$-fold metaplectic Brylinski--Deligne cover $\widetilde{\GL}_n(F)$, with central kernel $\mu_m(F)$, admits a splitting over the maximal compact subgroup $\GL_n(\mathcal O_F)$.

We emphasize that this splitting property is specific to this setting and does not extend to arbitrary metaplectic covers (see \cite[Section 1.5]{PATNAIK2017875} and \cite[Section 4]{GanGao2014}).

\medskip
For $n=2$, it is well known that $\GL_2(F)$ acts transitively on the set of vertices of the Bruhat--Tits tree $X$ associated with $\GL_2(F)$. As a consequence, all vertices are equivalent under this action. This transitivity implies that any maximal compact subgroup of $\GL_2(F)$, arising as the stabilizer of a vertex in $X$, is conjugate to any other. In particular, up to conjugacy, there is a unique maximal compact subgroup, namely $\GL_2(\mathcal O_F)$, and so $X \simeq \GL_2(F)/\GL_2(\mathcal O_F)$. 

Since the $m$-fold metaplectic Brylinski--Deligne cover $\widetilde{\GL}_2(F)$ admits a splitting over the maximal compact subgroup $\GL_2(\mathcal O_F)$, we can apply the results from Section \ref{twisted_cocycle} and aim to construct an analog of the Bruhat--Tits tree $X$ of $\GL_2(F)$ on which the metaplectic cover $\widetilde{\GL}_2(F)$ acts by automorphisms.

\begin{assumption} \label{Assumption1}
We work in a more general setting in which $G$ is a locally compact group acting vertex-transitively by automorphisms on a locally finite, infinite tree $X$. This implies that $X$ must be a regular tree. Let $K \subset G$ denote the stabilizer of a vertex of $X$. By vertex-transitivity, all such stabilizers are conjugate, and hence $K$ is, up to conjugacy, a maximal compact subgroup of $G$, in particular, $G/K \simeq X$. Let $T$ be a finite abelian group, and let $\alpha : G \times G \to T$ be a cocycle. We assume that $\alpha$ is trivial on $K \times K$, that is, $\alpha|_{K \times K} = 1$. Under this assumption, the subgroup $K$ splits over the central extension $\widetilde{G} := G \times_{\alpha} T$. We also fix a cross-section $\theta : G/K \to G \text{ such that }\theta(x)\,K = x$, and associate with it the cocycle $\sigma : G \times (G/K) \to K$ as in (\ref{sig}). As explained above, our goal is to construct an analog of the tree $X$ on which the covering group $\widetilde{G}$ acts by automorphisms.
\end{assumption}

\begin{remark}
\label{rem::compact_open_cover}
Under \textbf{Assumption  \eqref{Assumption1}}, the maximal compact-open subgroup $K \leq G$ splits in the metaplectic cover $\widetilde{G}$. Since splitting is inherited by subgroups, it follows that every compact-open subgroup $U \leq K$ also admits a splitting.

Moreover, since $T$ is finite (hence discrete), we have
\[
\widetilde{U} := U \times_{\alpha} T \simeq U \times T,
\]
and in particular $\widetilde{U}$ is locally compact for every compact-open subgroup $U \leq K$.

Finally, since $G$ admits a neighborhood basis of the identity consisting of compact-open subgroups contained in $K$, it follows that the identity of $\widetilde{G}$ also admits a neighborhood basis of the form $U \times \{e\} \leq U \times T$, with $U \subset G$ open. Consequently, $\widetilde{G}$ is endowed with a locally compact topology inherited from $G$.
\end{remark}

\medskip
To do that, we first introduce some notation. Let $Y$ be a graph, and denote by $V(Y)$ its set of vertices. For any $x, y \in V(Y)$, we write
\[
x \stackrel{e}{\sim} y
\]
if $x$ and $y$ are connected by an edge of unit length. In this case, we say that $x$ and $y$ are neighbors in $Y$. We define a metric $d$ on $V(Y)$ as follows. For $x, y \in V(Y)$, set $d(x,y) = n$ if:
\begin{itemize}
    \item there exists a sequence of distinct vertices
    \[
    x_0 = x, x_1, \ldots, x_n = y
    \]
    such that $x_k \stackrel{e}{\sim} x_{k+1}$ for all $0 \leq k \leq n-1$, and
    \item for any other sequence of vertices
    \[
    y_0 = x, y_1, \ldots, y_m = y
    \]
    satisfying $y_k \stackrel{e}{\sim} y_{k+1}$, one has $m \geq n$.
\end{itemize}

In other words, $d(x,y)$ is the length of the shortest path between $x$ and $y$.

\medskip
For the tree $X$ where $G$ acts by automorphisms and for the abelian finite group $T$, we let $\widetilde{X}$ be graph with vertices
\[
V(\widetilde{X}): =  \{v_0, v_{\rho} \; \vert \; v_0 := v \in V(X), \rho \in T\}
\]
and edges, all of unit length, placed as follows:
For $v, w \in V(X)$,  let $v_0 \e w_0$ in $\widetilde{X}$ whenever $v\e w$ in $X$, and for any $v\in V(X)$, and $\rho \in T$, let $v_{\rho} \e v_0$ in $\widetilde{X}$. By its definition, it can be easily noticed that $\widetilde{X}$ is a connected, locally finite infinite tree.

The group $\widetilde{G} = G\times_{\alpha} T$ acts on $\widetilde{X}$ as follows:
For $v \in V(X)$ 
and $\tilde{g} = (g,t) \in \widetilde{G}$,
\begin{enumerate} 
    \item[i)] 
    define  
    $\tilde{g} (v_{0}) := w_0$ if $g(v) = w$, \label{act_i}
    \item[ii)]
    for any $\rho \in T$, define
\begin{equation}
\label{action}
\tilde{g}(v_{\rho}) := w_{ \varepsilon \rho }
\quad \text{if} 
\quad
\text{
$\sigma^{(\alpha)}((g,t), v) =  \varepsilon$
and $w=g(v)$,} 
\end{equation}
where we identify $G/K$ with the vertices of $X$.
\end{enumerate}

By Corollary \ref{cor::corollary_twisted_action}, the action of $\widetilde{G}$ on $\widetilde{X}$ defined in (\ref{action}) is well defined and defines a genuine group action. Moreover, $\widetilde{X}$ may be viewed as $X \times_{\alpha,\theta} T$, which is used in Corollary \ref{cor::corollary_twisted_action}, although only at the level of vertices.

\begin{lemma} 
\label{lem::action_tilde_G}
In the action defined in (\ref{action}), $\widetilde{G}$ acts on the connected tree $\widetilde{X}$ by automorphisms.
\end{lemma}
\begin{proof} We verify that the action of $\widetilde{G}$ preserves the tree structure of $\widetilde{X}$. Let $v \in V(X)$ and $\tilde{g} \in \widetilde{G}$.

First, the $\widetilde{G}$-action preserves the valence of each vertex of $\widetilde{X}$. Indeed, by the definition of the action given in \eqref{action}, the valence of the vertex $\tilde{g}(v_0)$ is equal to the valence of $g(v)$ plus the cardinality of $T$. Moreover, for any $\rho \in T$, the valence of $v_\rho$ is one. By the same definition \eqref{action}, together with Corollary \ref{cor::corollary_twisted_action}, the valence of $\tilde{g}(v_\rho)$ remains one.

Next, we will verify that $\tilde{G}$ acts by isometries 
with respect to the metric $d$ on the tree $\widetilde{X}$.
For every  
$v, w \in V(X)$ and $\tilde{g}\in \tilde{G}$,
the fact that $G$ acts on $X$ by isometries implies that
\begin{equation} \label{v0}
d(\tilde{g}v_0, \tilde{g}w_0) = d(v_0,w_0)    
\end{equation}
 in $\tilde{X}$.
Noting that 
$\tilde{g}(v_0) = (g(v))_0$ and that 
\[
d(x_{\rho},x_0)=1\]
for all $ \rho \in T$ and $x\in V(X)$, it follows that
\begin{align*}
    d(\tilde{g}(v_{\rho}), \tilde{g}(w_{\zeta})) &=
d(\tilde{g}(v_{\rho}), \tilde{g}(v_0)) +d(\tilde{g}(v_0),\tilde{g}(w_0))+ d(\tilde{g}(w_0), \tilde{g}(w_{\zeta}))\\
&= d\Big( (gv)_{\sigma^{(\alpha)}(\tilde{g}, v)\rho},
    (gv)_0 \Big) 
    + d\Big( \tilde{g}(v_0), \tilde{g}(w_0) \Big) 
    + d\Big( (gw)_0, (gw)_{\sigma^{(\alpha)}(\tilde{g}, w)\zeta} \Big) &\\
    &= 1 + d\Big( \tilde{g}(v_0), \tilde{g}(w_0) \Big)  + 1 \\
    &= d(v_{\rho}, v_0) + d(v_0,w_0) + d(w_0, w_{\zeta})&\\
    &= d(v_{\rho}, w_{\zeta})
\end{align*}
for all $\rho, \zeta \in T$.
The verification that $d(\tilde{g}(v_\rho), \tilde{g}(w_0)) = d(v_\rho, w_0)$ is similar, and the lemma is proved. 
\end{proof}

%\begin{lemma} $\widetilde{G}$ acts transitively on the vertices of $\widetilde{X}$.
%\end{lemma}
%\begin{proof}
%Let $x, y\in V(\widetilde{X})$. There are  vertices $v_0,w_0 \in V(X)$ such that $x \in \{v_{0}, v_{\rho} \;\vert \; \rho \in T \}$ and $y \in \{w_{0}, w_{\rho} \;\vert \; \rho \in T \}$. 
%Since $G$ acts transitively on $V(X)$, choose $g\in G$ such that $gv_0=w_0$.
%Then, it suffices to notice that if $y = w_{\delta}$ for some $\delta\in T$ and if $\varepsilon \in T$ is such that $\sigma^{(\alpha)}((g,t), v) =  \varepsilon$, then 
%by Corollary \ref{cor::corollary_twisted_action} we must  have $(g,t) (v_{\varepsilon^{-1}\delta}) = w_{\delta}$.
%\end{proof}

Given an apartment $\mathscr{A}$ in $X$, we define the corresponding apartment $\widetilde{\mathscr{A}}$ in $\widetilde{X}$ to be the subgraph whose vertices are
\[
\{\, v_0, v_\rho \mid \rho \in T \,\},
\]
as $v_0=v$ ranges over the vertices of $\mathscr{A}$.

\section{Gelfand pairs}
\label{Gelfand_pairs}
Let $G$ be a locally compact group and let $K$ be a compact subgroup of $G$.
The spherical Hecke algebra
$\mathscr{H}(G, K)$ is the space of continuous, compactly supported $K$-biinvariant functions $C_{c}(K\backslash G /K)$
with convolution 
\[
\psi * \phi (x) = \int_G \psi(xg)\phi(g^{-1}) dg, \quad \psi, \phi \in \mathscr{H}(G,K).
\]
The pair $(G,K)$ is a Gelfand pair if the Hecke algebra
$\mathscr{H}(G,K)$ is commutative.

Gelfand pairs are of relevance in representation theory due to the fact that $(G,K)$ is a Gelfand pair if and only if all irreducible unitary representations $\pi \in \widehat{G}$ satisfy \[
\dim \pi^K \leq 1,\] 
where $\pi^K$ is the subspace of $K$-fixed vectors. 

\medskip
Given a connected, locally finite infinite tree $X$ and a fixed vertex $v_0$ of $X$, a \emph{geodesic ray} in $X$ with base point $v_0$ is, by definition, an infinite geodesic in the tree $X$ that starts at $v_0$. The equivalence class (or ``endpoint at infinity'') of such a geodesic ray is called an \emph{end} of the tree $X$. The set of all ends determined by geodesic rays in $X$ with base point $v_0$ is denoted by $\partial X$ and is called the \emph{(visual) boundary} of the tree $X$. It is easy to see that $\partial X$
does not depend on the chosen base point $v_0$.

\begin{remark}
\label{rem::visual_boundary}
Under \textbf{Assumption \eqref{Assumption1}}, and keeping the notation introduced there and at the end of Section \ref{sec::connected_tree}, note that the visual boundary $\partial \widetilde{X}$ of $\widetilde{X}$ coincides with the visual boundary $\partial X$ of $X$.
\end{remark}

%Suppose the locally compact group $G$ acts on a locally finite building $X$. Let $\partial X$ be the virtual boundary of $X$ equipped with the cone topology. 

 By a more general theorem that holds in the context of affine buildings (see \cite[Theorem 1.1]{CaCi15}), one can extract the following lemma in the setting of locally finite trees. 
 
\begin{proposition}[See Propositions 2.4 and 2.6 in \cite{Ci16}]
\label{prop::Gelfand_pair}
    Let $d \in \mathbb{N}$ with $d > 2$, and let $\mathscr{T}_d$ be a $d$-regular tree. Let $G$ be a closed, non-compact, locally compact group such that $G \leq \operatorname{Aut}(\mathscr{T}_d)$ acts cocompactly and type-preservingly on $\mathscr{T}_d$. Let $x_0$ be a vertex of $\mathscr{T}_d$, and let $K := \operatorname{Stab}_G(x_0)$ be the stabilizer of $x_0$ in $G$.

Then $(G,K)$ is a Gelfand pair if and only if $K$ acts transitively on $\partial \mathscr{T}_d$; that is, for any two distinct ends $\xi_1, \xi_2$, there exists $k \in K$ such that $k(\xi_1) = \xi_2$.
\end{proposition}

At the heart of the proof of Proposition \ref{prop::Gelfand_pair} lies the following lemma, which provides a necessary condition for $(G,K)$ to be a Gelfand pair.

\begin{lemma} \label{doublecoset}
 Let $G$ be a locally compact group and $K$ be a compact subgroup of $G$. 

If $g^{-1} \in K g K$ for all $g \in G$, then $(G,K)$ is a Gelfand pair.
\end{lemma}

\begin{proof}
For every $\phi \in \mathscr{H}(G,K)$ and $x\in G$ 
note that if $x^{-1} \in KxK$ 
then $\phi(x) = \phi(x^{-1})$. 
As a consequence,
for all $\phi, \psi \in \mathscr{H}(G,K)$ and $x\in G$ we have
\begin{align*}
    (\phi\ast \psi)(x) &= \int_G \phi(xg) \psi(g)dg\\
    &= \int_G  \psi(x^{-1} h)\phi(h)  dh\\
    &= (\psi \ast \phi)(x^{-1})\\
    &=(\psi \ast \phi)(x).
\end{align*}

\end{proof}

For the rest of this section, we work under \textbf{Assumption \eqref{Assumption1}}, keeping the notation introduced there. 

\begin{assumption} \label{Assumption2}
Fix a vertex $x \in X$ and set $K := \operatorname{Stab}_G(x)$, the pointwise stabilizer of $x$ in $G$. Let $\ell$ be a bi-infinite geodesic line in $X$ such that $x$ lies on $\ell$. Suppose that there exists an element $a \in G$ such that $a(\ell) = \ell$ setwise and $a$ acts as a translation along $\ell$ of translation length $1$; that is,
$d_X(y, a(y)) = 1$, for every vertex $y \in \ell$. Define $A := \langle a^n \mid n \in \mathbb{Z} \rangle$,
which is an infinite abelian subgroup of $G$.    
\end{assumption}

\begin{lemma}
\label{lem::tilde_KAK}
Under \textbf{Assumptions \eqref{Assumption1}} and
\textbf{\eqref{Assumption2}}, and keeping the notation introduced there, suppose that $G$ admits the polar decomposition $G = K A K$. Then
\[
\widetilde{G} = \widetilde{K} \, \widetilde{A} \, \widetilde{K},
\]
where $\widetilde{A} = A \times_{\alpha} T$ and $\widetilde{K} = K \times_{\alpha} T$ denote the preimages of $A$ and $K$, respectively, in $\widetilde{G}$.
\end{lemma}

\begin{proof}
By Lemma \ref{lem::action_tilde_G}, we know that $\widetilde{G}$ acts on $\widetilde{X}$ by automorphisms. Moreover, by \textbf{Assumption \eqref{Assumption1}}, the group $K$ splits in $\widetilde{G}$, and hence $\widetilde{K} = K \times T$ as a direct product. In addition, by the definition of the action of $\widetilde{G}$ on $\widetilde{X}$, it is easy to see that
\[
\widetilde{K} = \operatorname{Stab}_{\widetilde{G}}(x_0),
\]
where $x_0$ denotes the lift of the vertex $x$ to $\widetilde{X}$.

Let $\tilde{g} = (g,t) \in \widetilde{G}$. If $\tilde{g}(x_0) = x_0$, then by the above, $\tilde{g} \in \widetilde{K}$.

Suppose now that $\tilde{g}(x_0) = y_0$, where $y \neq x$ is a vertex of $X$. Then we know that $g = k_1 \gamma k_2$, with $\gamma \in A$ and $k_1, k_2 \in K$. Choose any $\tilde{\gamma} = (\gamma, \xi) \in \widetilde{A}$ and $\tilde{k}_1 = (k_1, t_1) \in \widetilde{K}$. Then
\[
\tilde{\gamma}^{-1}\tilde{k}_1^{-1}\tilde{g}(x_0) = x_0,
\]
which implies that
\[
\tilde{\gamma}^{-1}\tilde{k}_1^{-1}\tilde{g} \in \widetilde{K}.
\]
Consequently,
\[
\tilde{g} \in \widetilde{K}\tilde{\gamma}\widetilde{K} \subset \widetilde{K}\widetilde{A}\widetilde{K}.
\]
\end{proof}

\begin{remark}
One can observe that $G = \mathrm{GL}_2(F)$, where $F$ is a non-Archimedean local field, and $K = \mathrm{GL}_2(\mathcal{O}_F)$ satisfy the hypotheses of Lemma \ref{lem::tilde_KAK}. In this case, one may take $A$ to be the diagonal subgroup of $\mathrm{GL}_2(F)$.
\end{remark}

Let $F$ be a non-Archimedean local field. Let $G$ be the group of $F$-points of a split reductive algebraic group $\mathscr{G}$ defined over $F$, and let $K = \mathscr{G}(\mathcal{O}_F)$. In \cite[Theorem 9.1]{mcnamara}, McNamara shows that $(\widetilde{G}, K)$ is a Gelfand pair (see also his definition from \cite[Section 9]{mcnamara} of the Hecke algebra). Moreover, in \cite[Section 3, Theorem 3.3(ii)]{Kazhdan-Patterson}, in the case of $G= \mathrm{GL}_n(F)$, Kazhdan and Patterson prove that $(\widetilde{G}, \widetilde{K})$ is a Gelfand pair by lifting the map $g \mapsto (g^{-1})^t$ to $\widetilde{G}$.

Here, we provide an alternative geometric proof for locally compact groups $G$ whose coverings $\widetilde{G}$ act on a tree $\widetilde{X}$ in a suitable way.

\begin{proposition} 
Under 
\textbf{Assumptions \eqref{Assumption1}} and \textbf{\eqref{Assumption2}}, and keeping the notation introduced there, suppose that $G$ admits the polar decomposition $G = K A K$ and $K$ acts transitively on $\partial X$. Then 
$(\widetilde{G}, \widetilde{K})$ is a Gelfand pair. 
\end{proposition}
\begin{proof}
 By Lemma \ref{doublecoset}, it suffices to show that $
\tilde{g}^{-1} \in \widetilde{K} \, \tilde{g} \, \widetilde{K}$, for all $\tilde{g} \in \widetilde{G}$.

Moreover, using the decomposition $\widetilde{G} = \widetilde{K} \, \widetilde{A} \, \widetilde{K}$ from Lemma \ref{lem::tilde_KAK}, it therefore suffices to prove that
\[
\tilde{\gamma}^{-1} \in \widetilde{K} \, \tilde{\gamma} \, \widetilde{K}
\quad \text{for every } \tilde{\gamma} \in \widetilde{A}.
\]

Therefore, take $\tilde{\gamma} \in \widetilde{A}$ and write $\tilde{\gamma} = (a^n, \xi)$ for some $n \in \Z$, and $\xi \in T$. Then $\tilde{\gamma}^{-1} = (a^{-n}, \xi')$ for some $\xi' \in T$ and hence 
 $\tilde{\gamma}^{-1}(x_0) = y_0$, where $y= a^{-n}(x)$ is a vertex of $X$, and where $x_0, y_0$ denote the lifts of the vertices $x, y$ to $\widetilde{X}$.

Since $K$ acts transitively on the visual boundary $\partial X$, there is some $k\in K$ such that 
\[
k a^{-n} (x) = a^n(x).
\]
Therefore $a^{-n} k a^{-n} (x) = x$ which means that $a^{-n} k a^{-n} \in K$.

Then, for any $\eta \in T$, taking $\tilde{k}: = (k,\eta)$, we get that $\tilde{k} \tilde{\gamma}^{-1}(x_0) = z_0$ in $\widetilde{X}$, where $z = a^{n}(x)$ in $X$. Hence 
 $ \tilde{\gamma}^{-1}\tilde{k} \tilde{\gamma}^{-1}(x_0)=x_0$, and so $\tilde{\gamma}^{-1}\tilde{k} \tilde{\gamma}^{-1} \in \widetilde{K}$, implying that $\tilde{\gamma}^{-1} \in \widetilde{K}\tilde{\gamma} \widetilde{K}$.
 %must be some $\xi'''\in T$such that $(a^n, \xi''')(k,\eta)\tilde{g}^{-1} \in \tilde{K}$. 
\end{proof}

\section{Contraction and parabolic subgroups for metaplectic covers} \label{relative}

Let us first introduce the needed definitions.

\begin{definition}
Let $G$ be a totally disconnected locally compact group and let $\beta = \{g_n\}_{n>0}$ be a sequence
of elements in $G$. We define the \textbf{positive parabolic subgroup} associated with $\beta$ by
\begin{align*}
P_\beta^+ 
  &:= \left\{ g \in G \,\middle|\, \{g_n^{-1} g g_n\}_{n \in \mathbb{N}} 
     \text{ is bounded in the topology of } G \right\}.
\end{align*}
Note that $P_\beta^+$ is a subgroup of $G$. We also define the \textbf{contraction subgroup}
\begin{align*}
U_\beta^+ 
  &:= \left\{ g \in G \,\middle|\, \lim_{n \to \infty} g_n^{-1} g g_n = e \right\}.
\end{align*}
Then $U_\beta^+$ is a subgroup of $G$, called the contraction group corresponding to $\beta$.
Since it need not be closed in general, we denote its closure by
\[
N_\beta^+ := \overline{U_\beta^+}.
\]

In the same way, but using $g_n g g_n^{-1}$, we define $P_\beta^-$, $U_\beta^-$, and $N_\beta^-$,
called the \textbf{negative parabolic} and \textbf{contraction subgroups} associated with $\beta$, respectively.
\end{definition}

As shown in \cite[Section 3]{BaWillis2004}, $P_\beta^+, P_\beta^-$ are closed subgroup of $G$. Moreover, from the algebraic point of view, we have the following Levi decomposition:

\begin{theorem}(see~\cite[Proposition 3.4, Corollary 3.17]{BaWillis2004}).
 Let $G$ be a totally disconnected, locally compact group which is also metrizable, and let $a \in G$. Take $\beta: = \{a^n\}_{n>0}$. Then
\[
P_\beta^+ = U_\beta^+ \, M_\beta,
\]
where
\[
M_\beta := P_\beta^+ \cap P_\beta^-.
\]
Moreover, $U_\beta^+$ is normal in $P_\beta^+$.   
\end{theorem} 

\begin{remark}
    For a reductive algebraic group over a local field $F$, the contraction subgroup of an inner automorphism defined by a semisimple element coincides with the group of $F$-points of the unipotent radical of the associated parabolic subgroup (see for example \cite[Corollary 4.29]{Cocio1}).
\end{remark}

For example, the next easy lemma describes the parabolic subgroups for covers $\widetilde{G} \leq \Aut(\widetilde{X})$.

\begin{lemma}
\label{lem::parabolic_metaplectic}
Under \textbf{Assumptions \eqref{Assumption1}} and \textbf{\eqref{Assumption2}}, and keeping the notation introduced there, suppose that $G$ admits the polar decomposition $G = K A K$. Let $\xi_+$ and $\xi_-$ be the endpoints of the bi-infinite geodesic line $\ell \subset X$, which, respectively, are the attracting and repelling fixed points of $a \in A$. Let $\beta := \{a^n\}_{n>0} \subset A$ and $\widetilde{\beta} := \{(a^n, t_n)\}_{n>0} \subset \widetilde{A}$. Then
\[
P_{\widetilde{\beta}}^+ = P_{\beta}^+ \times_\alpha T   = \widetilde{G}_{\xi_+} = G_{\xi_+} \times_\alpha T  \subset \widetilde{G},
 \quad
P_{\widetilde{\beta}}^- = P_{\beta}^- \times_\alpha T   = \widetilde{G}_{\xi_-} = G_{\xi_-} \times_\alpha T  \subset \widetilde{G}.
\]
Moreover, $U_\beta^\pm \times_\alpha T$ is a subgroup of $P_{\widetilde{\beta}}^\pm$. 
\end{lemma}

\begin{proof}
Recall that the locally compact topology on $G$ induces a locally compact topology on $\widetilde{G}$ (see Remark~\ref{rem::compact_open_cover}), and that $\partial \widetilde{X} = \partial X$ (see Remark~\ref{rem::visual_boundary}).

Take $\tilde{g} = (g,t) \in P_{\widetilde{\beta}}^+$. Since $T$ is finite (hence discrete, and therefore bounded), it suffices to consider the subset $\{a^{-n} g a^n\}_{n >0}$ of $G$. By \cite[Proposition 4.15]{Cocio1}, we obtain that $g \in P_{\beta}^+ = G_{\xi_+}$. Since $\partial \widetilde{X} = \partial X$, it follows that $\tilde{g} \in \widetilde{G}_{\xi_+}$. Hence,
\[
P_{\widetilde{\beta}}^+ \subseteq \widetilde{G}_{\xi_+}.
\] 

Now take $\tilde{g} = (g,t) \in \widetilde{G}_{\xi_+}$. Using again the fact that $\partial \widetilde{X} = \partial X$, it follows that $g \in P_{\beta}^+ = G_{\xi_+}$. Thus, the sequence $\{a^{-n} g a^n\}_{n>0}$ is bounded in $G$, and consequently the sequence $\{(a^n,t_n)^{-1} \tilde{g} (a^n,t_n)\}_{n>0}$ is bounded in $\widetilde{G}$. This shows that
\[
\widetilde{G}_{\xi_+} \subseteq P_{\widetilde{\beta}}^+.
\]
 From the above arguments (i.e. \cite[Proposition 4.15]{Cocio1}, $\partial \widetilde{X} = \partial X$, and $T$ is discrete and bounded), it is easy to conclude that $P_{\widetilde{\beta}}^\pm = P_{\beta}^\pm\times_\alpha T$, $\widetilde{G}_{\xi_\pm} = G_{\xi_\pm} \times_\alpha T$, and that  $U_\beta^\pm \times_\alpha T$ is a subgroup of $P_{\widetilde{\beta}}^\pm$.
\end{proof}

\begin{remark}
    By computing the conjugation
\[
(a^n,t_n)^{-1}\,\tilde g\,(a^n,t_n)
\]
using the defining cocycle $\alpha$, one sees that, in general, the contraction subgroups $U_{\widetilde{\beta}}^{\pm}$ might not contain $U_\beta^{\pm} \times \{e\}$. In particular, the contraction subgroups $U_\beta^\pm$ do not necessarily split over general covers $\widetilde{G}$.

By contrast, for the standard metaplectic covers of $\mathrm{GL}_n(F)$, it is well known that unipotent subgroups split. In particular, the unipotent
radical of a parabolic subgroup - which coincides with the contraction subgroup associated to a suitable semisimple element - admits a canonical splitting over the metaplectic cover (see \cite[Proposition 4.1]{mcnamara}).
\end{remark}

\bibliographystyle{amsplain} 
\bibliography{references.bib}

@incollection {mcnamara,
    AUTHOR = {McNamara, P. J.},
     TITLE = {Principal series representations of metaplectic groups over
              local fields},
 BOOKTITLE = {Multiple {D}irichlet series, {L}-functions and automorphic
              forms},
    SERIES = {Progr. Math.},
    VOLUME = {300},
     PAGES = {299--327},
 PUBLISHER = {Birkh\"auser/Springer, New York},
      YEAR = {2012},
       DOI = {10.1007/978-0-8176-8334-4\_13},
       URL = {https://doi.org/10.1007/978-0-8176-8334-4_13},
}

@article {Kazhdan-Patterson,
    AUTHOR = {Kazhdan, D. A. and Patterson, S. J.},
     TITLE = {Towards a generalized {S}himura correspondence},
   JOURNAL = {Adv. in Math.},
  FJOURNAL = {Advances in Mathematics},
    VOLUME = {60},
      YEAR = {1986},
    NUMBER = {2},
     PAGES = {161--234},
      ISSN = {0001-8708},
   MRCLASS = {22E55 (11F70 11R39 22E50)},
  MRNUMBER = {840303},
MRREVIEWER = {Jean-Luc\ Brylinski},
       DOI = {10.1016/S0001-8708(86)80010-X},
       URL = {https://doi.org/10.1016/S0001-8708(86)80010-X},
}

@article {Cocio1,
    AUTHOR = {Ciobotaru, C.},
     TITLE = {A unified proof of the {H}owe--{M}oore property},
  FJOURNAL = {Journal of Lie Theory},
    VOLUME = {25},
      YEAR = {2015},
    NUMBER = {1},
     PAGES = {65--89}
}

@book {gelbart,
    AUTHOR = {Gelbart, S. S.},
     TITLE = {Weil's representation and the spectrum of the metaplectic
              group},
    SERIES = {Lecture Notes in Mathematics},
    VOLUME = {Vol. 530},
 PUBLISHER = {Springer-Verlag, Berlin-New York},
      YEAR = {1976},
     PAGES = {140}
}

@article{PATNAIK2017875,
title = {On {I}wahori--{W}hittaker functions for metaplectic groups},
journal = {Advances in Mathematics},
volume = {313},
pages = {875-914},
year = {2017},
issn = {0001-8708},
doi = {https://doi.org/10.1016/j.aim.2017.04.005},
url = {https://www.sciencedirect.com/science/article/pii/S0001870815301699},
author = {Patnaik, M. and Puskás, A.},
}

@article{GanGao2014,
  author       = {Gan, Wee Teck and Gao, Fan},
  title        = {The {L}anglands--{W}eissman Program for {B}rylinski--{D}eligne Extensions},
  journal      = {arXiv preprint},
  eprint       = {1409.4039},
  archivePrefix= {arXiv},
  primaryClass = {math.NT},
  year         = {2014},
  doi          = {10.48550/arXiv.1409.4039},
}

@article{CaCi15, 
title={Gelfand pairs and strong transitivity for {E}uclidean buildings}, 
volume={35}, 
DOI={10.1017/etds.2013.102}, 
number={4}, 
journal={Ergodic Theory and Dynamical Systems}, 
author={Caprace, P.-E. and Ciobotaru, C.}, 
year={2015}, 
pages={1056–1078}
}

@article{Ci16,
    author = {Ciobotaru, C.},
    title = {The universal group of {B}urger--{M}ozes and the {H}owe--{M}oore property},
    journal = {arXiv:1612.09427},
    year = {2016}
}

@article{BaWillis2004,
  author  = {Baumgartner, U. and Willis, G. A.},
  title   = {Contraction groups and scales of automorphisms of totally disconnected locally compact groups},
  journal = {Israel Journal of Mathematics},
  volume  = {142},
  pages   = {221--248},
  year    = {2004},
  doi     = {10.1007/BF02771534}
}

@article{ash-doud,
    author = "Ash, A. and Doud, D.",
    title = "Even {G}alois representations and the cohomology of {GL}(2,$\mathbb{Z}$)",
    journal = "Annales Mathématiques Québec",
    volume = "43",
    year = "2019",
    pages = "1-35"
}

\end{document}